\documentclass[b5paper]{mcom-l}

\usepackage{mathtools,bm} 
\usepackage{array,booktabs} 
\usepackage{graphicx} 
\usepackage{url,bookmark} 

\usepackage[textwidth=\the\textwidth,textheight=\the\textheight,hmarginratio=1:1,vmarginratio=3:2]{geometry}

\newcommand\includeT[1]{\raisebox{-1.7ex}{\includegraphics{HePe-Ts-#1}}}
\newcommand\Ti   {\includeT{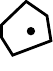}}
\newcommand\Tii  {\includeT{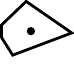}}
\newcommand\Tiii {\includeT{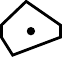}}
\newcommand\Tiv  {\includeT{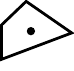}}
\newcommand\Tv   {\includeT{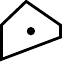}}
\newcommand\Tvi  {\includeT{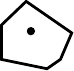}}
\newcommand\Tvii {\includeT{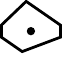}}
\newcommand\Tviii{\includeT{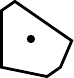}}
\newcommand\Tix  {\includeT{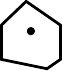}}
\newcommand\Tx   {\includeT{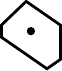}}
\newcommand\Txi  {\includeT{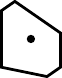}}
\newcommand\Txii {\includeT{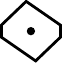}}
\newcommand\Txiii{\includeT{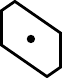}}
\newcommand\Txiv {\includeT{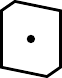}}

\newcommand\includeD[1]{\raisebox{-0.5\height}{\includegraphics{HePe-Ds-#1}}}

\newcommand\includeV[2]{\begin{tabular}[b]{c}
	\includegraphics[scale=1.5]{HePe-Vs-#1}\\
	$x'\in[#2)$
	\end{tabular}\allowbreak\ignorespaces
}

\newcommand\includeW[2]{\begin{tabular}[b]{c}
	\includegraphics[scale=1.5]{HePe-Ws-#1}\\
	$x'\in[#2)$
	\end{tabular}\allowbreak\ignorespaces
}

\DeclarePairedDelimiterX\set[2]\lbrace\rbrace{#1\bm:#2}
\DeclarePairedDelimiter\qfl\lfloor\rfloor
\DeclarePairedDelimiter\abs\lvert\rvert
\DeclareMathOperator\Fin{Fin} 
\DeclareMathOperator\Pal{Pal} 
\newcommand*\T{\mathcal{T}} 
\newcommand*\calP{\mathcal{P}} 
\newcommand*\C{\mathbb{C}} 
\newcommand*\N{\mathbb{N}} 
\newcommand*\R{\mathbb{R}} 
\newcommand*\Q{\mathbb{Q}} 
\newcommand*\Z{\mathbb{Z}} 
\newcommand*\conj[1]{\overline{#1}} 
\newcommand*\defined[1]{\emph{#1}} 
\newcommand*\ii{\mathrm{i}} 
\newcommand*\Deltax{\Delta^{\!*}} 
\newcommand*\Lx{L^{\mkern-1mu*}} 
\newcommand*\eqdef{\coloneqq} 
\newcommand*\defeq{\eqqcolon} 
\newcommand*\degree{\ensuremath{{}^\circ}} 

\def\romanenumi{\def\theenumi{\roman{enumi}}}
\def\Alphenumi{\def\theenumi{\Alph{enumi}}}

\theoremstyle{plain}
\newtheorem{lemma}{Lemma}[section]
\newtheorem{proposition}[lemma]{Proposition}
\newtheorem{theorem}[lemma]{Theorem}
\newtheorem{corollary}[lemma]{Corollary}

\theoremstyle{definition}

\newtheorem{algor}[lemma]{Algorithm}
\newtheorem{example}[lemma]{Example}

\theoremstyle{remark}
\newtheorem{remark}[lemma]{Remark}

\numberwithin{equation}{section}

\usepackage{hyperref}

\begin{document}

\title{Spectral properties of cubic complex Pisot units}

\author[T. Hejda]{Tom\'a\v s Hejda}
\address{Dept.\@ of Mathematics FNSPE, Czech Technical University in Prague, Trojanova~13, Prague 12000, Czech Rep.}
\curraddr{LIAFA, CNRS UMR 7089, Universit\'e Paris Diderot -- Paris 7, Case 7014, 75205 Paris Cedex 13, France}
\email{tohecz@gmail.com}

\author[E. Pelantov\'a]{Edita Pelantov\'a}
\address{Dept.\@ of Mathematics FNSPE, Czech Technical University in Prague, Trojanova~13, Prague 12000, Czech Rep.}
\email{edita.pelantova@fjfi.cvut.cz}


\begin{abstract}
For a real number $\beta>1$, Erd\H os, Jo\'o and Komornik study distances between consecutive points in the set
\[
	X^m(\beta)=\Bigl\{\sum_{j=0}^n a_j \beta^j \bm: n\in\N,\,a_k\in\{0,1,\dots,m\}\Bigr\}
.\]
Pisot numbers play a crucial role for the properties of $X^m(\beta)$.
Following the work of Za\"imi, who considered $X^m(\gamma)$ with $\gamma\in\mathbb{C}\setminus\mathbb{R}$ and $|\gamma|>1$,
 we show that for any non-real $\gamma$ and $m<|\gamma|^2-1$, the set $X^m(\gamma)$ is not relatively dense in the complex plane.

Then we focus on complex Pisot units $\gamma$ with a positive real conjugate $\gamma'$ and $m>|\gamma|^2-1$.
If the number $1/\gamma'$ satisfies Property~(F), we deduce that $X^m(\gamma)$
 is uniformly discrete and relatively dense, i.e., $X^m(\gamma)$ is a Delone set.
Moreover, we present an algorithm for determining two parameters of the Delone set $X^m(\gamma)$
 which are analogous to minimal and maximal distances in the real case $X^m(\beta)$.
For $\gamma$ satisfying $\gamma^3+\gamma^2+\gamma-1=0$, explicit formulas for the two parameters are given.
\end{abstract}

\maketitle

\section{Introduction}\label{sect:intro}

In \cite{erdos_joo_komornik_1990,erdos_joo_komornik_1998}, Erd\H os, Jo\'o and Komornik studied the set
\[
	X^m(\beta)\eqdef \set[\Big]{\sum_{j=0}^n a_j\beta^j}{n\in\N,\,a_k\in\{0,1,\dots,m\}}
,\]
 where $\beta>1$.
Since this set has no accumulation points, we can find an increasing sequence
\[
	0=x_0<x_1<x_2<\dotsb<x_k<\dotsb
\]
such that $X^m(\beta)=\set{x_k}{k\in\N}$.
The research of Erd\H os et al.\@ aims to describe distances between consecutive points of $X^m(\beta)$, i.e.,
 the sequence $(x_{k+1}-x_k)_{k\in\N}$.
The properties of this sequence depend on the value $m\in\N$.
It is easy to show that when $m\geq\beta-1$, we have $x_{k+1}-x_k\leq1$ for all $k\geq0$;
 and when $m<\beta-1$, the distances $x_{k+1}-x_k$ can be arbitrarily large.

Also, many properties of $X^m(\beta)$ depend on $\beta$ being a \defined{Pisot number}
 (i.e., an algebraic integer $>1$ such that all its Galois conjugates are $<1$ in modulus).
Bugeaud~\cite{bugeaud_1996} showed that
\[
	\ell_m(\beta)\eqdef \liminf_{k\to\infty} (x_{k+1}-x_k)>0
\quad\text{for all $m\in\N$}
\]
 if and only if base $\beta$ is a Pisot number.
Recently, Feng \cite{feng_2013} proved a stronger result that the bound $\beta-1$ for the alphabet size is crucial.
In particular, $\ell_m(\beta)=0$ if and only if $m>\beta-1$ and $\beta$ is not a Pisot number.

Therefore, the case $\beta$ Pisot and $m>\beta-1$ has been further studied.
From the approximation property of Pisot numbers we know that for a fixed $\beta$ and $m>\beta-1$
 the sequence $(x_{k+1}-x_k)$ takes only finitely many values.
Feng and~Wen~\cite{feng_wen_2002} used this fact to show that the sequence of distances $(x_{k+1}-x_k)$ is substitutive:
 roughly speaking, it can be generated by a system of rewriting rules over a finite alphabet.
This allows us, for a fixed $\beta$ and $m$, to determine values of all distances $(x_{k+1}-x_k)$ and subsequently the value of $\ell_m(\beta)$.
An algorithm for obtaining the minimal distance $\ell_m(\beta)$ for certain $\beta$
 was also proposed by Borwein and Hare~\cite{borwein_hare_2002}.

The first formula which determines the value of $\ell_m(\beta)$ for all $m$ at once appeared
 in 2000: Komornik, Loreti and Pedicini \cite{komornik_loreti_pedicini_2000} studied the base Golden mean.
The generalization of this result to all quadratic Pisot units was provided by Takao Komatsu \cite{komatsu_2002} in 2002.

To the best of our knowledge, the value of
\begin{equation}\label{eq:Lmb}
	L_m(\beta)\eqdef \limsup_{k\to\infty} (x_{k+1}-x_k)
\end{equation}
 for all $m$ is only known for the base Golden mean, due to Borwein and Hare~\cite{borwein_hare_2003}.
Of course, for a given $m$, the value of $L_m(\beta)$ can be computed using \cite{feng_wen_2002}.

Za\"imi \cite{zaimi_2004} was interested in a complementary question:
Fix the alphabet size, i.e., the maximal digit $m$, and look for the extreme values of $\ell_m(\beta)$,
 where $\beta$ runs through Pisot numbers in $(m,m+1)$.
He showed that $\ell_m(\beta)$ is maximized for certain quadratic Pisot numbers.

Besides that, Za\"imi started the study of the set $X^m(\gamma)$,
 where $\gamma$ is a complex number $>1$ in modulus, and he put
\begin{equation}\label{eq:ell}
	\ell_m(\gamma)\eqdef \inf\set[\big]{\abs{x-y}}{x,y\in X^m(\gamma),\,x\neq y}
.\end{equation}
He proved an analogous result to the one for real bases by Bugeaud, namely that
$\ell_m(\gamma)>0$ for all $m$ if and only if $\gamma$ is a complex Pisot number,
 where a \defined{complex Pisot number} is defined as a non-real algebraic
 integer $>1$ in modulus whose Galois conjugates except its complex conjugate are $<1$ in modulus.

In the complex plane, $\ell_m(\gamma)$ and $L_m(\gamma)$ cannot be defined as simply as in the real case since we have
 no natural ordering of the set $X^m(\gamma)$ in $\C$.
To overcome this difficulty, we were inspired by notions used in the definition of Delone sets.
We say that a set $\Sigma$ is:
\begin{itemize}
\item \defined{uniformly discrete} if there exists $d>0$ such that $\abs{x-y}\geq d$ for all distinct $x,y\in\Sigma$;
\item \defined{relatively dense} if there exists $D>0$ such that for all $x\in\C$ the closed ball
 $B(x,D/2)=\set{z\in\C}{\abs{z-x}\leq D/2}$ contains a point from $\Sigma$.
\end{itemize}
A set that is both uniformly discrete and relatively dense is called a \defined{Delone set}.

Clearly, if $\ell_m(\gamma)$ as given by \eqref{eq:ell} is positive, then $X^m(\gamma)$ is uniformly discrete and $\ell_m(\gamma)$
 is the maximal $d$ in the definition of uniform discreteness.
Hence $X^m(\gamma)$ is uniformly discrete for all $m$, when $\gamma$ is a complex Pisot number.

Let us define
\[
	L_m(\gamma)\eqdef \inf\set[\big]{D>0}{B(x,D/2)\cap X^m(\gamma)\neq\emptyset \text{ for all $x\in\C$}}
.\]
In particular, $L_m(\gamma)=+\infty$ if and only if $X^m(\gamma)$ is not relatively dense.

The question for which pairs $(\gamma,m)$ the set $X^m(\gamma)$ is uniformly discrete or is relatively dense is far from being solved.
We provide a necessary condition for relative denseness and we show that in certain cases, it is sufficient as well:

\begin{theorem}\label{thm:m2}
Let $\gamma\in\C$ be a non-real number $>1$ in modulus.
\begin{enumerate}\romanenumi

\item
 If $m<\abs\gamma^2-1$, then $X^m(\gamma)$ is not relatively dense.

\item
 \cite{zaimi_2004} If $m>\abs\gamma^2-1$ and $\gamma$ is not an algebraic number,
 then $X^m(\gamma)$ is not uniformly discrete.

\end{enumerate}
\end{theorem}

The aim of this article is to study the sets $X^m(\gamma)$ simultaneously for all $m\in\N$,
 for a certain class of cubic complex Pisot units with a positive conjugate $\gamma'$.
For such $\gamma$ the R\'enyi expansions in base $\beta\eqdef1/\gamma'$ have nice properties,
 which will be crucial in the proofs.
When this base satisfies a certain finiteness property, called Property~(F) in~\cite{akiyama_2000}, we show that
 for all sufficiently large $m$ the set $X^m(\gamma)\subseteq\C$ is a cut-and-project set;
 roughly speaking, $X^m(\gamma)$ is formed by projections of points from the lattice $\Z^3$
 which lie in a sector bounded by two parallel planes in $\R^3$; see Theorem~\ref{thm:CnP}.
From that, the asymptotic behaviour of $\ell_m(\gamma)$ and $L_m(\gamma)$ follows easily, namely:
\begin{equation}\label{eq:sqrt-m}
	\ell_m(\gamma)=\Theta(1/\sqrt m)
\quad\text{and}\quad
	L_m(\gamma)=\Theta(1/\sqrt m)
,\end{equation}
where $f(m)=\Theta(1/\sqrt m)$ means that $K_1/\sqrt m\leq f(m)\leq K_2/\sqrt m$ for some positive constants $K_1$, $K_2$.

The method of inspection of Voronoi cells for a specific cut-and-project set,
 as established by Mas\'akov\'a, Patera and Zich \cite{masakova_patera_zich_2003_i,masakova_patera_zich_2003_ii,masakova_patera_zich_2005},
 enables us to give a general formula
 for both $\ell_m(\gamma)$ and $L_m(\gamma)$.
In the case where $\gamma=\gamma_T\approx-0.771+1.115\ii$ is the complex Tribonacci constant, i.e.,
 the complex root of $Y^3+Y^2+Y-1$ with a positive imaginary part, we get the following result:

\begin{theorem}\label{thm:tribo}
Let $\gamma$ be a complex root of the polynomial $Y^3+Y^2+Y-1$,
 $m\in\N$, and $k\in\Z$ be the greatest integer such that $m\geq(1-\gamma')\bigl(\frac{1}{\gamma'}\bigr)^k$,
 where $\gamma'$ is the real Galois conjugate of $\gamma$.
Then we have
\begin{equation}\label{eq:tribo}
	\ell_m(\gamma)=\abs{\gamma}^{-k}
\quad\text{and}\quad
	L_m(\gamma)=2\sqrt{\frac{1-\smash{(\gamma')^2}}{3-\smash{(\gamma')^2}}}\abs{\gamma}^{3-k}
.\end{equation}
\end{theorem}

\medskip

The article is organized as follows.
In Section~\ref{sect:beta}, we recall certain notions from the theory of $\beta$-expansions.
Section~\ref{sect:pf:m2} provides the proof of the 1st part of Theorem~\ref{thm:m2}.
In Section~\ref{sect:CnP} we prove that $X^m(\gamma)$ is a cut-and-project set in certain cases.
Section~\ref{sect:voronoi} describes the algorithms for computing $\ell_m(\gamma)$ and $L_m(\gamma)$.
These algorithms are applied to the complex Tribonacci number in Section~\ref{sect:tribo},
 providing the proof of Theorem~\ref{thm:tribo}.
In Section~\ref{sect:delone} we compute another characteristic of $X^m(\gamma)$ that is based on Delone tessellations.
Comments and open problems are in Section~\ref{sect:concl}.

All computations were carried out in Sage~\cite{sage}.
The pictures were drawn using Ti\textit kZ~\cite{tikz}.

\section{Preliminaries}\label{sect:beta}

\subsection{\texorpdfstring{$\beta$}{beta}-numeration}

Let us recall some facts concerning $\beta$-expansions.
For a real base $\beta>1$, and for a number $x\geq0$, there exist a unique $N\in\Z$
 and unique integer coefficients $a_{N},a_{N-1},a_{N-2},\dotsc$ such that $a_N\neq0$ and 
\[
	0\leq x-\sum_{j=n}^N a_j\beta^j <\beta^n
	\quad\text{for all $n\leq N$}
.\]
The string $a_N a_{N-1} \dotsm a_1 a_0.a_{-1}a_{-2}\dotsm$ is then called the \defined{R\'enyi expansion} of $x$
 in base~$\beta$~\cite{renyi_1957}.
We immediately see that $a_j\in\set{k\in\Z}{0\leq k<\beta}$.
For $\beta\notin\Z$, it means that $a_j\in\{0,\dots,\qfl\beta\}$,
 where $\qfl\beta$ denotes the greatest integer $\leq\beta$.
If only finitely many $a_j$\!'s are non-zero, we speak about the \defined{finite R\'enyi expansion} of~$x$.
The set of numbers $x\in\R$ such that $\abs x$ has a finite R\'enyi expansion is denoted $\Fin(\beta)$.
We say that $\beta>1$ satisfies \defined{Property~(F)} if $\Fin(\beta)$ is a ring,
 i.e., $\Fin(\beta)=\Z[1/\beta]$, where $\Z[y]$ denotes as usual the integer combinations of powers of~$y$.

\subsection{Complex Pisot numbers}

We widely use the algebraic properties of a cubic complex Pisot number $\gamma$.
Such a number has two other Galois conjugates.
One of them is the complex conjugate $\conj\gamma$.
The second one is real and $<1$ in modulus; we denote it $\gamma'$;
 we have either $-1<\gamma'<0$ or $0<\gamma'<1$.
In general, for $z\in\Q(\gamma)$ we denote by $z'\in\Q(\gamma')\subset\R$ its image under
 the Galois isomorphism that maps $\gamma\mapsto\gamma'$.
When $\gamma$ is a unit (i.e., the constant term of its minimal polynomial is $\pm1$),
 we know that $\Z[1/\gamma]=\Z[\gamma]=\gamma\Z[\gamma]$.

The method we present here can be applied only in the case when:
\begin{equation}\label{eq:cmplx-gamma}
	\text{\begin{tabular}[t]{c}
		$\gamma$ is a cubic complex Pisot unit, its real Galois conjugate $\gamma'$ is positive,
	\\
		and $\beta\eqdef 1/\gamma'$ has Property (F).
	\end{tabular}}
\end{equation}
It implies that the minimal polynomial of $\gamma$ is of the form $Y^3+bY^2+aY-1$ with $a,b\in\Z$.
Such a polynomial has a complex root if and only if its discriminant is negative, i.e.,
\[
	-18ab-4a^3+a^2b^2+4b^3-27<0
.\]

The number $\beta=1/\gamma'$ is a root of $Y^3-aY^2-bY-1$.
Akiyama \cite{akiyama_2000} showed that $\beta$ has Property~(F) if and only if
\[
	\abs{b-1}\leq a
\quad\text{and}\quad
	b\geq-1
.\]
Therefore we are interested in cases where both conditions are satisfied.

In particular, the \defined{complex Tribonacci constant} $\gamma_T\approx -0.771+1.115\ii$ (the root of $Y^3+Y^2+Y-1$
 with a positive imaginary part)
 satisfies \eqref{eq:cmplx-gamma}, as well as the complex roots of polynomials $Y^3+bY^2+aY-1$ for $b=0,\pm1$ and $a\geq1$, with the exception $(a,b)=(1,-1)$.

\section{Proof of Theorem~\ref{thm:m2}}\label{sect:pf:m2}

We prove the first part of Theorem~\ref{thm:m2}.
We cannot easily follow the lines of the proof of the result for the real case (i.e., that $m<\beta-1$ implies $L^m(\beta)=+\infty$),
 because it relies on the natural ordering of $\R$.
In the proof of the theorem, the following `folklore' lemma about the asymptotic density of relatively dense sets is used:

\begin{lemma}\label{lem:f}
Let $\Sigma\subset\C$ be a relatively dense set.
Then
\begin{equation}\label{eq:fr}
	\liminf_{r\to\infty} \frac{\#\bigl(\Sigma\cap B(0,r)\bigr)}{r^2} >0
,\end{equation}
where $\#A$ is the number of elements of the set $A$.
\end{lemma}

\begin{proof}
Since $\Sigma$ is relatively dense, there exists $\lambda>0$ such that every square in $\C$ with side $\lambda$
 contains a point of $\Sigma$.
Therefore every cell of the lattice $\lambda\Z[\ii]=\set{\lambda a+\ii\lambda b}{a,b\in\Z}$ contains a point of $\Sigma$.
Since $B(0,r)$ contains at least $n^2$ cells, where $n=\qfl[\big]{r\sqrt2/\lambda}$, we get
\[
	\liminf_{r\to\infty} \frac{\#\bigl(\Sigma\cap B(0,r)\bigr)}{r^2}
	\geq \liminf_{r\to\infty} \frac{\qfl[\big]{r\sqrt2/\lambda}^2}{r^2} = \frac{2}{\lambda^2}
	>0
.\qedhere\]
\end{proof}

\begin{proof}[Proof of Theorem~\ref{thm:m2}, 1st statement]
For simplicity, we denote $\Sigma\eqdef X^m(\gamma)$.

First, we show that for any $r\geq m$ we have
\[
	\Sigma\cap B\bigl(0,\abs\gamma r-m\bigr)\subseteq \gamma\bigl(\Sigma\cap B(0,r)\bigr)+\{0,\dots,m\}
\]
 and therefore
\begin{equation}\label{eq:frL}
	\# \bigl(\Sigma\cap B(0,\abs\gamma r-m)\bigr)\leq (m+1)\#\bigl(\Sigma\cap B(0,r)\bigr)
.\end{equation}
To prove this, consider $x=\sum_{j=0}^k a_j\gamma^j$ with $a_j\in\{0,\dots,m\}$
 and such that $\abs x\leq\abs\gamma r-m$.
Then $y\eqdef (x-a_0)/\gamma=\sum_{j=1}^k a_j\gamma^{j-1}\in \Sigma$ and
 $\abs y\leq (\abs x+a_0)/\abs\gamma\leq (\abs\gamma r-m+m)/\abs\gamma=r$.
Since $x=\gamma y+a_0$, the inclusion is valid.

Our aim is to prove that under the assumption $m<\abs\gamma^2-1$, the set $\Sigma$ is not relatively dense.
According to Lemma~\ref{lem:f}, it is enough to construct a sequence $(r_k)$ such that $r_k\to\infty$ and
\[
	\lim_{k\to\infty} n_k=0,
\quad\text{where}\quad
	n_k\eqdef\frac{\#\bigl(\Sigma\cap B(0,r_k)\bigr)}{r_k^2}
.\]
Since $\Sigma=X^m(\gamma)$ always contains $0$, the set $\Sigma\cap B(0,r_k)$ is non-empty
 and we have that $n_k>0$.

Consider a sequence given by the recurrence relation $r_{k+1}= \abs\gamma r_k-m$
 and $r_0\eqdef \abs\gamma^2+\frac{m}{\abs\gamma-1}>m$.
The choice of $r_0$ guarantees that $r_k=\abs\gamma^{k+2}+\frac{m}{\abs\gamma-1}$,
 therefore $r_k\to\infty$ and $r_{k+1}/r_k\to\abs\gamma$.
Then \eqref{eq:frL} gives 
\(
\# \bigl(\Sigma\cap B(0,r_{k+1})\bigr)\leq (m+1)\#\bigl(\Sigma\cap B(0,r_k)\bigr)
\),
which yields
\[
	\frac{n_{k+1}}{n_k}
	\leq \frac{(m+1)r_k^2}{r_{k+1}^2}
	\xrightarrow{k\to\infty}
	\frac{m+1}{\abs\gamma^2}<1
,\]
 therefore $n_k\to0$ as desired.
\end{proof}

\section{Cut-and-project sets versus \texorpdfstring{$X^m(\gamma)$}{Xm(gamma)}}\label{sect:CnP}

A cut-and-project scheme in dimension $d+e$ consists of two linear maps
$\Psi:\R^{d+e}\to\R^d$ and $\Phi:\R^{d+e}\to\R^e$ satisfying:
\begin{enumerate}
\item $\Psi(\R^{d+e}) = \R^d$ and the restriction of $\Psi$ to the lattice $\Z^{d+e}$ is injective;
\item the set $\Phi(\Z^{d+e})$ is dense in $\R^e$.
\end{enumerate}

Let $\Omega\subset\R^e$ be a nonempty bounded set such that its closure equals the closure of its interior,
 i.e., $\overline\Omega=\overline{\Omega^\circ}$.
Then the set
\[
	\Sigma(\Omega)\eqdef\set[\big]{\Psi(v)}{v\in\Z^{d+e}, \Phi(v)\in\Omega}\subseteq\R^d
\]
 is called a \defined{cut-and-project} set with acceptance window $\Omega$.
Cut-and-project sets can be defined in a slightly more general way, cf.~\cite{moody_1997}.
The assumptions made on the acceptance window $\Omega$ ensure that every cut-and-project set is a Delone set.

We use the concept of cut-and-project sets for $d=2$ and $e=1$.
With a slight abuse of notation, we consider $\Psi:\R^3\to\C\simeq\R^2$.
Then it is straightforward that for a cubic complex Pisot number $\gamma$, the set defined by
\begin{equation}\label{eq:SO}
	\Sigma_\gamma(\Omega)=\set[\big]{z\in\Z[\gamma]}{z'\in\Omega}
,\quad\text{where $\Omega\subseteq\R$ is an interval}
,\end{equation}
 is a cut-and-project set.
Really, we have
\begin{align*}
	\Psi_\gamma(v_0,v_1,v_2)&=v_0+v_1\gamma+v_2\gamma^2
		\simeq  \binom{\Re(v_0+v_1\gamma+v_2\gamma^2)}{\Im(v_0+v_1\gamma+v_2\gamma^2)}
\\ \text{and}\quad
	\Phi_\gamma(v_0,v_1,v_2)&=v_0+v_1\gamma'+v_2(\gamma')^2
.\end{align*}
We will omit the index $\gamma$ in the sequel.
We now show how $X^m(\gamma)$ fit into the cut-and-project scheme:

\begin{theorem}\label{thm:CnP}
Let $\gamma$ be a cubic complex Pisot unit with a positive conjugate $\gamma'$, and let $m$ be an integer $m\geq\abs\gamma^2-1$.
Suppose that base $1/\gamma'$ has Property~(F).
Then $X^m(\gamma)$ is a cut-and-project set, namely
\begin{equation}\label{eq:CnP}
	X^m(\gamma)=\Sigma(\Omega)=\set[\big]{z\in\Z[\gamma]}{z'\in\Omega}
\quad\text{with }
	\Omega=\bigl[0,m/(1-\gamma')\bigr)
.\end{equation}
\end{theorem}

\begin{proof}
Inclusion $\subseteq$:
Let $z\in X^m(\gamma)$. Then $z=\sum_{j=0}^n a_j\gamma^j$ with $a_j\in\{0,\dots,m\}$ and clearly $z\in\Z[\gamma]$.
Moreover,
\[
	0\leq z'=\sum_{j=0}^n a_j (\gamma')^j\leq \sum_{j=0}^n m(\gamma')^j<\frac{m}{1-\gamma'}
.\]

Inclusion $\supseteq$:
Let us take $z\in\Z[\gamma]$ with $z'\in\Omega$.
Denote $\beta=1/\gamma'=\gamma\conj\gamma=\abs\gamma^2$.
We discuss the following two cases:

\begin{enumerate}

\item
 Suppose $0\leq z'<1$.
 The real base $\beta$ has Property~(F) by the hypothesis.
 Therefore every number from $\Z[1/\beta]\cap[0,1)$
  has a finite expansion $0.a_1a_2a_3\dots a_n$ over the alphabet $\{0,\dots,m_0\}$, where $m_0\eqdef \qfl{\beta}$
  (the expansion certainly starts after the fractional point since $z<1$).
 This means that $z'=\sum_{j=1}^n a_j \beta^{-j}$ and therefore $z=\sum_{j=1}^n a_j \gamma^j\in X^{m_0}(\gamma)$.
 Since $X^{m_0}(\gamma)\subseteq X^m(\gamma)$, we get $z\in X^m(\gamma)$.

\item
 Suppose $1\leq z'<m/(1-\gamma')$.
 Since $z'<\sum_{j=0}^\infty m\beta^{-j}$, there exists a~minimal $k\geq0$ such that $z'-\sum_{j=0}^k m\beta^{-j}<0$.
 Let $b\in\{0,\dots,m\}$ be such that
 \[
	0\leq z'-\sum_{j=0}^{k-1}m\beta^{-j}-b\beta^{-k}<\beta^{-k}
 ,\]
  where $\sum_{j=0}^{-1}m\beta^{-j}\eqdef0$.
 Then
 \[
	u'\eqdef \beta^k\biggl(z'-\sum_{j=0}^{k-1} m\beta^{-j}-b\beta^{-k}\biggr)
 \]
  satisfies $0\leq u'<1$, and by the previous case there exist $a_1,\dots,a_n\in\{0,\dots,m_0\}$
  such that $u'=\sum_{j=1}^n a_j\beta^{-j}$.
 Altogether,
 \[
	z'=\sum_{j=0}^{k-1} m(\gamma')^j + b(\gamma')^k + \sum_{j=k+1}^{k+n} a_{j-k}(\gamma')^j
 \]
 and $z\in X^m(\gamma)$.\qedhere

\end{enumerate}
\end{proof}

The property of cut-and-project sets which allows us to determine the values of $\ell_m(\gamma)$ and $L_m(\gamma)$
 is the self-similarity.
We say that a Delone set $\Sigma\subseteq\C$ is \defined{self-similar} with a factor $\kappa\in\C$, $\abs\kappa>1$,
 if $\kappa\Sigma\subseteq\Sigma$.
In general, cut-and-project sets are not self-similar.
In our special case \eqref{eq:SO}, not only the sets are self-similar, but we can prove even a stronger
 property that will be useful later:

\begin{proposition}\label{prop:self-sim}
Let $\gamma$ be a cubic complex Pisot unit.
Then
\[
	\Sigma\bigl((\gamma')^k\Omega\bigr)=\gamma^k\Sigma(\Omega)
\quad\text{for any interval $\Omega$ and any $k\in\Z$}
.\]
In particular, if $\Omega=[0,c)$ and $\gamma'$ is positive, then $\gamma'\Omega\subseteq\Omega$ and $\gamma\Sigma\subseteq\Sigma$.
\end{proposition}

\begin{proof}
We prove the claim for $k=\pm1$, the general case follows by induction.
Because $\Z[\gamma]=\gamma\Z[\gamma]$, we have that
\begin{multline}\label{eq:proof-self-sim}
	\Sigma(\gamma'\Omega)
	=\set[\big]{x\in\gamma\Z[\gamma]}{x'\in\gamma'\Omega}
	=\set[\big]{x\in\gamma\Z[\gamma]}{\tfrac{1}{\gamma'}x'\in\Omega}
\\
	=\gamma\set[\big]{y\in\Z[\gamma]}{y'\in\Omega}
	=\gamma\Sigma(\Omega)
,\end{multline}
 which implies the validity of the statement for $k=+1$.
If we apply \eqref{eq:proof-self-sim} to the window $\tilde\Omega=\gamma'\Omega$, we get
 $\Sigma(\tilde\Omega)=\gamma\Sigma(\frac{1}{\gamma'}\tilde\Omega)$, i.e., 
 $\frac{1}{\gamma}\Sigma(\tilde\Omega)=\Sigma(\frac{1}{\gamma'}\tilde\Omega)$,
 which implies the validity of the statement for $k=-1$.
\end{proof}

\begin{remark}
Theorem~\ref{thm:CnP} and Proposition~\ref{prop:self-sim} imply the asymptotic behaviour of $\ell_m(\gamma)$ and $L_m(\gamma)$
 as described in \eqref{eq:sqrt-m}, because $\abs{\gamma'}=1/\sqrt{\abs\gamma}$.
\end{remark}

\section{Voronoi tessellations}\label{sect:voronoi}

For a Delone set $\Sigma$, the \defined{Voronoi cell} of a point $x\in\Sigma$ is the set of points
 which are closer to $x$ than to any other point in $\Sigma$.
Formally
\begin{equation}\label{eq:def-T}
	\T(x)\eqdef\set[\big]{z\in\C}{\abs{z-x}\leq\abs{z-y}\text{ for all }y\in\Sigma}
.\end{equation}
The cell is a convex polygon having $x$ as an interior point.
Clearly $\bigcup_{x\in\Sigma} \T(x)=\C$ and the interiors of two cells do not intersect.
Such a collection of cells $\set{\T(x)}{x\in\Sigma}$ is called a \defined{tessellation} of the complex plane.
For every cell $\T(x)$ we define two characteristics:
\begin{itemize}

\item $\delta(\T(x))$ is the maximal diameter $d>0$ such that $B(x,d/2)\subseteq\T(x)$;

\item $\Delta(\T(x))$ is the minimal diameter $D>0$ such that $\T(x)\subseteq B(x,D/2)$.

\end{itemize}
These $\delta$ and $\Delta$ allow us to compute the values of $\ell_m(\gamma)$ and $L_m(\gamma)$, namely
\[
	\ell_m(\gamma)=\inf_x \delta\bigl(\T(x)\bigr)
\quad\text{and}\quad
	L_m(\gamma)=\sup_x \Delta\bigl(\T(x)\bigr)
,\]
where $x$ runs the whole set $\Sigma=X^m(\gamma)$.

A \defined{protocell} of a point $x$ is the set $\T(x)-x$.
We can define $\delta,\Delta$ analogously for the protocells.
The set of all protocells of the tessellation of $\Sigma$ is called the \defined{palette}
 of $\Sigma$.
We therefore obtain that
\begin{equation}\label{eq:dl}
	\ell_m(\gamma)=\inf_{\T} \delta(\T)
\quad\text{and}\quad
	L_m(\gamma)=\sup_{\T} \Delta(\T)
,\end{equation}
 where $\T$ runs the whole palette of $\Sigma$.

For computing $\delta(\T)$ and $\Delta(\T)$, we modify the approach of~\cite{masakova_patera_zich_2003_i},
 where \mbox{$2$-dimensional} cut-and-project sets based on quadratic irrationalities are concerned.
To find the Voronoi cell of a point $x\in\Sigma(\Omega)$ one does not need to consider all points $y\in\Sigma(\Omega)$.
It is easy to see that only points $y$ closer to $x$ than $\Delta(\T(x))$ influence the shape of the tile $\T(x)$, i.e.,
\begin{equation}\label{eq:XY}
	\T(x)=\set[\big]{z\in\C}{\abs{z-x}\leq\abs{z-y}\text{ for $y\in\Sigma(\Omega)$, $\abs{y-x}\leq\Delta(\T(x))$}}
.\end{equation}

But before the shape of $\T(x)$ is known, we do not know the value of $\Delta(\T(x))$.
So we need to find some positive constant $L$ such that
\begin{equation}\label{eq:good-L}
	\Delta\bigl(\T(y)\bigr)\leq L
\quad\text{for all}\quad
	y\in\Sigma(\Omega)
.\end{equation}

In the rest of this section, we consider cut-and-project sets $\Sigma(\Omega)$ as given by \eqref{eq:SO},
 where $\gamma$ satisfies \eqref{eq:cmplx-gamma},
 i.e., $1/\gamma'$ has Property~(F),
 and where $\Omega=[0,c)$ with $c>0$ (however, not necessarily of the form $c=\frac{m}{1-\gamma'}$).
We denote by $\Re z=\frac{z+\conj z}{2}$ and $\Im z=\frac{z-\conj z}{2\ii}$ respectively
 the real and the imaginary part of $z\in\C$.

\begin{lemma}\label{lem:L}
Let $\Omega=[0,c)$ be an interval.
Let $p$ be the first positive integer such that $\Im(\gamma^p)$ and $\Im\gamma$ have the opposite signs
 and let $k$ be the smallest integer satisfying $(\gamma')^k< c/2$.
Then
\begin{equation}\label{eq:L}
	L\eqdef\abs\gamma^k \! \max_{\substack{i,j\in\{0,p-1,p\}\\i<j}}
	\abs[\Big]{\frac{
		\gamma^{ i + j } (\gamma^ i -\gamma^ j )
	}{
		\Im(\gamma^ i \conj\gamma^ j )
	}}
\end{equation}
 satisfies $\Delta(\T(y))\leq L$ for all $y\in\Sigma(\Omega)$.
\end{lemma}

\begin{figure}[t]
\includegraphics{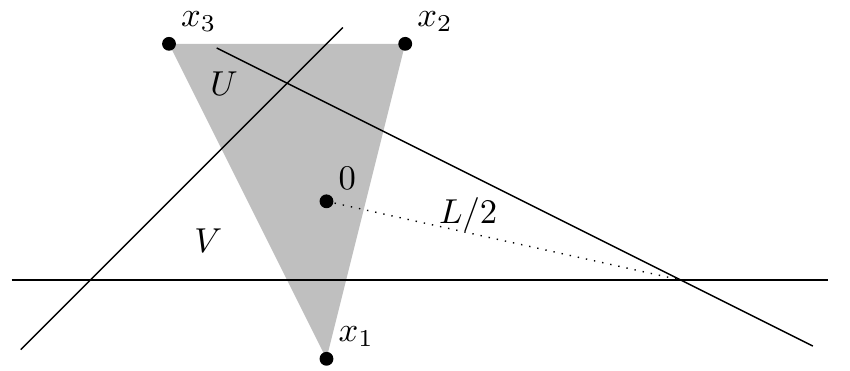}
\caption{To the proof of Lemma~\ref{lem:L}.}
\label{fig:triangle}
\end{figure}

\begin{proof}
We first prove the statement for $y=0$.
The choice of $k$ guarantees that $x_1\eqdef\gamma^k$, $x_2\eqdef\gamma^{k+p-1}$ and $x_3\eqdef\gamma^{k+p}$
 satisfy $x_1,x_2,x_3\in\Sigma(\Omega)$,
 whereas the choice of $p$ guarantees that $0$ is an inner point of the triangle $U$ with
 vertices $x_1$, $x_2$, $x_3$ (see Figure~\ref{fig:triangle}).
According to \eqref{eq:def-T} we have
\[
	V\eqdef\set[\big]{z\in\C}{\abs{z-0}\leq\abs{z-x_j}\text{ for $j=1,2,3$}} \supseteq \T(0)
.\]
Let $\rho$ be the radius of the smallest ball centered at $0$ and containing the whole triangle $V$.
From the definition of $\T(x)$ and $\Delta(\T(x))$ we see that $\Delta(\T(0))\leq 2\rho$.

The vertices of $V$ are the points $v_{12}, v_{23}, v_{31}$ such that
\begin{equation}\label{eq:vij-cond}
	\abs{x_i-v_{ij}}=\abs{x_j-v_{ij}}=\abs{0-v_{ij}}
.\end{equation}
These equations have a unique solution
\begin{equation}\label{eq:vij-xij}
	v_{ij} = \ii\frac{x_ix_j(\conj{x_i}-\conj{x_j})}{2\Im(x_i\conj{x_j})},
\quad\text{whence}\quad
	 \abs{v_{ij}} = \frac{1}{2} \abs[\Big]{\frac{x_ix_j(x_i-x_j)}{\Im(x_i\conj{x_j})}}
.\end{equation}
Then $\rho=\max\,\abs{v_{ij}}$, thus the estimate \eqref{eq:L} is valid for $y=0$
 and it remains to show that it is valid for all $y\in\Sigma(\Omega)$.
If $y'\in[0,c/2)$ then the three points $y+x_j$ for $j=1,2,3$ are in $\Sigma(\Omega)$.
If $y'\in[c/2,c)$ then the three points $y-x_j$ for $j=1,2,3$ are in $\Sigma(\Omega)$.
Both of these cases follow from the fact that $x_1',x_2',x_3'\in(0,c/2)$.
Therefore either $x_1,x_2,x_3$ or $-x_1,-x_2,-x_3$ are elements of $\Sigma(\Omega)-y$,
 which means that the same estimate \eqref{eq:L} can be used.
\end{proof}

To describe the palette of $\Sigma(\Omega)$, we find all possible $L$-patches, i.e., the local configurations around the points of $\Sigma(\Omega)$ up to a distance $L$.
More precisely, the \defined{$L$-patch} of $x\in\Sigma(\Omega)$ is the set
\begin{equation}\label{eq:AB}
	\calP_L(x)\eqdef \bigl(\Sigma(\Omega)\cap B(x,L)\bigr) - x
.\end{equation}
Since we consider the window $\Omega=[0,c)$, the $L$-patch equals
\begin{equation}\label{eq:LC}
	\calP_L(x)=\set{z\in\Z[\gamma]}{x'+z'\in[0,c)\text{ and }\abs z\leq L}
.\end{equation}

\begin{lemma}\label{lem:L-eq}
Let $x,y\in\Sigma(\Omega)$ with $\Omega=[0,c)$ and $L$ satisfying \eqref{eq:good-L}.
Then the equality of two $L$-patches $\calP_L(x)=\calP_L(y)$
  implies the equality of the protocells, i.e.,  $\T(x)-x=\T(y)-y$.
\end{lemma}

\begin{proof}
 Using \eqref{eq:XY} we can write
\[
	\T(x) = \set[\big]{z\in\C}{\abs{z-x}\leq\abs{z-v}\text{ for all $v\in\Sigma(\Omega)\cap B(x,L)$}}
\]
 and thus
\[
	\T(x)-x = \set[\big]{s\in\C}{\abs{s}\leq\abs{s-w}\text{ for all $w\in\calP_L(x)$}}
,\]
 which depends only on $\calP_L(x)$ and not on $x$ itself.
\end{proof}

\begin{lemma}\label{lem:Lc}
Let $x,y\in\Sigma(\Omega)$ with $\Omega=[0,c)$ and $L>0$.
If $\calP_L(x)\neq \calP_L(y)$ then there exists $\xi$ from the following finite subset of $[0,c]$:
\begin{align}\label{eq:Xi}
	\Xi\eqdef
	\set[\big]{z'}{z\in\calP_L(0)}
	\cup
	\set[\big]{c-z'}{z\in\calP_L(0)}
,\end{align}
 such that $\xi$ lies between $x'$ and $y'$, more precisely, $\min\{x',y'\}<\xi\leq\max\{x',y'\}$.
\end{lemma}

\begin{proof}
 Without loss of generality, suppose that there exists $z$ such that $z\in\calP_L(x)$ and $z\notin\calP_L(y)$.
 According to \eqref{eq:LC} we have $\abs z\leq L$, $x'+z'\in[0,c)$, and $y'+z'\notin[0,c)$.

If $x'<y'$ then $x'+z'<c\leq y'+z'$, therefore $0\leq x'<c-z'\leq y'<c$ and thus $x'$ and $y'$ are separated by $\xi\eqdef c-z'$.
 We have that $c-z'\in(0,c)$, or equivalently $z'\in(0,c)$.
 As $\abs z\leq L$, we conclude that $z\in\calP_L(0)$.

If $x'>y'$ then $y'+z'<0\leq x'+z'$, therefore $0\leq y'<-z'\leq x'<c$ and thus $x'$ and $y'$ are separated by $\xi\eqdef-z'$.
We have that $-z'\in(0,c)$.
As $\abs{-z}=\abs{z}\leq L$, we conclude that $-z\in\calP_L(0)$.
\end{proof}

The two lemmas enable us to partition the interval $\Omega$ into sub-intervals such that 
 the points of $\Sigma(\Omega)$ whose Galois conjugates lie in the same sub-interval have the same protocell,
 formally:

\begin{corollary}\label{corol:Xi}
Let $\Omega=[0,c)$ be an interval.
Then there exists a finite set $\Xi=\{\xi_0=0<\xi_1<\dots<\xi_{N-1}<\xi_N=c\}$ such that
 the mapping
\[
	x'\mapsto \T(x)-x
\]
 is constant on $[\xi_{j-1},\xi_j)\cap\Z[\gamma']$ for each $j=1,\dots,N$.
\end{corollary}

\begin{proof}
Consider $L$ satisfying \eqref{eq:good-L} and let $\Xi$ be given by \eqref{eq:Xi}.
Suppose $x,y\in\Sigma(\Omega)$ satisfy $x',y'\in[\xi_{j-1},\xi_j)$.
According to Lemma~\ref{lem:Lc} we have $\calP_L(x)=\calP_L(y)$.
Therefore by Lemma~\ref{lem:L-eq} their protocells are equal.
\end{proof}

\begin{remark}
From the last two statements, we can conclude that $\Sigma([0,c))$ is a \defined{repetitive set with finite local complexity},
 i.e., for each $L>0$, the number of $L$-patches is finite and each of them appears for infinitely many $x\in\Sigma([0,c))$.
Finite local complexity is justified by the finiteness of set $\Xi$.
Repetitiveness is justified by the fact that the each interval $[\xi_{j-1},\xi_j)$
 contains infinitely many points of $\Z[\gamma']$.

Let us mention that while finite local complexity is a property of all cut-and-project sets,
 repetitiveness depends on the boundary of window $\Omega$.
In particular, cut-and-project sets with $\Omega$ of the form $[l,r)$ are repetitive,
 cf.~\cite{moody_1997}.
\end{remark}

The corollary is constructive and it allows us to compute all protocells of the Voronoi
 tessellation of $\Sigma(\Omega)$ for a fixed $\Omega=[0,c)$:

\begin{algor}\label{algor:Xi}\leavevmode
\begin{enumerate}
\item[\textbullet]
 Input: $\gamma$ satisfying \eqref{eq:cmplx-gamma}, $\Omega=[0,c)$, $L$ satisfying \eqref{eq:good-L},
  e.g.\@ given by \eqref{eq:L}.
\item[\textbullet]
 Output: The palette of $\Sigma(\Omega)$.
\item Compute the set $\Xi=\{\xi_0=0<\xi_1<\dots<\xi_{N-1}<\xi_N=c\}$ given by \eqref{eq:Xi}.
\item\label{enum:Xi-local} For each interval $[\xi_j,\xi_{j+1})$ compute the corresponding $L$-patch.
\item\label{enum:Xi-proto} Compute the corresponding protocells to each of these patches.
\item Remove possible duplicates in the list of protocells.
\end{enumerate}
\end{algor}

\begin{figure}
\includeV{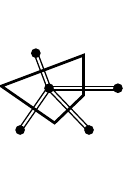}{0,1}
\includeV{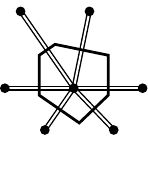}{1,1+\gamma'}
\includeV{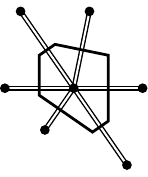}{1+\gamma',\frac{1}{\gamma'}}
\includeV{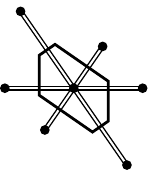}{\frac{1}{\gamma'},2+\gamma'}
\includeV{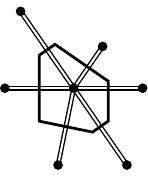}{2+\gamma',1+\frac{1}{\gamma'}}
\includeV{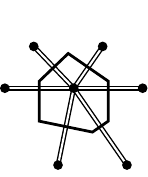}{1+\frac{1}{\gamma'},\frac{1}{{\gamma'}^2}}
\includeV{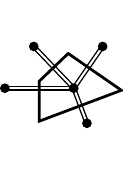}{\frac{1}{{\gamma'}^2},1+\frac{1}{{\gamma'}^2}}
\caption{Voronoi protocells (the palette) for
 $X^2(\gamma)=\Sigma(\Omega)$, where $\Omega=[0,\frac{2}{1-\gamma'})$
 and $\gamma=\gamma_T$ is the complex Tribonacci constant.
}\label{fig:one-tribo}
\end{figure}

\begin{example}\label{ex:one-tribo}
We illustrate how the algorithm works for $\gamma=\gamma_T$ the complex Tribonacci constant and
 $c=2/(1-\gamma')=\beta^2+1$, where we denote as usual $\beta\eqdef 1/\gamma'$.
In this case, $\Sigma([0,c))=X^2(\gamma)$ by Theorem~\ref{thm:CnP}.
We have $k=-1$ in Lemma~\ref{lem:L} and since $\arg\gamma\in(\pi/2,\pi)$, we have $p=2$.
Therefore $L$ is the maximum of the values
\[
	\frac{1}{\abs\gamma}\abs[\Big]{\frac{\gamma(\gamma-1)}{\Im\gamma}}
	\approx 1.877
,\quad
	\frac{1}{\abs\gamma}\abs[\Big]{\frac{\gamma^2(\gamma^2-1)}{\Im(\gamma^2)}}
	\approx 1.877
,\quad
	\frac{1}{\abs\gamma}\abs[\Big]{\frac{\gamma^2(\gamma-1)}{\Im\gamma}}
	\approx 2.546
,\]
 i.e., $L=\abs{\gamma(\gamma-1)}/\Im\gamma$.
The set $\set{z'}{z'\in\Z[\gamma']\cap[0,c)\text{ and }\abs z\leq L}$ contains $28$ points.
The set $\Xi$, given as a union of two $28$-element sets in \eqref{eq:Xi}, has only $33$ elements
 instead of $56$ because many elements appear in both of them.
This gives $32$ cases in steps \ref{enum:Xi-local}--\ref{enum:Xi-proto} of the algorithm.
After we remove the duplicates in the list of the $32$ protocells,
 we end up with the list in Figure \ref{fig:one-tribo}.
The double lines connect the center of the protocell with the centers of the neighboring cells.
A part of the Voronoi tessellation of $\Sigma(\Omega)$ is drawn in Figure~\ref{fig:voronoi}.
Note that all computations are performed in the algebraic library of Sage~\cite{sage}.
Numbers $a+b\gamma+c\gamma^2\in\Z[\gamma]$ are stored as triples of integers $(a,b,c)$
 and thus results of all arithmetic operations are precise.

Let us determine the parameters $\ell_2(\gamma)$ and $L_2(\gamma)$, with the help of relations~\eqref{eq:dl}.
For each protocell $\T$, the value $\delta(\T)$ is by definition the length of the shortest double line
 in the picture of $\T$.
In Figure~\ref{fig:T1}, the 1st protocell is depicted: the neighbors are (counterclockwise) $x_1=1$, $x_2=2+2\gamma+\gamma^2=\gamma^{-2}$,
 $x_3=1+\gamma+\gamma^2=\gamma^{-1}$ and $x_4=2+\gamma+\gamma^2=1+\gamma^{-1}$.
The closest point of these to $0$ is $x_2=\gamma^{-2}$.
For the last protocell, the closest point is analogously $-\gamma^2$.
Therefore $\delta(\T)=\abs{\gamma^{-2}}=\gamma'$ for the first and the last protocell.
For the rest of the protocells, the closest point to $0$ is $\pm(1+\gamma+\gamma^2)=\pm\gamma^{-1}$,
 and therefore $\delta(\T)=\abs{\gamma^{-1}}=\sqrt{\gamma'}=1/\sqrt\beta$.
Since $\ell_2(\gamma)$ is the minimum of all $\delta(\T)$, we get that
\[
	\ell_2(\gamma)=\gamma'\approx0.544
.\]

\begin{figure}
\includegraphics[width=0.95\linewidth]{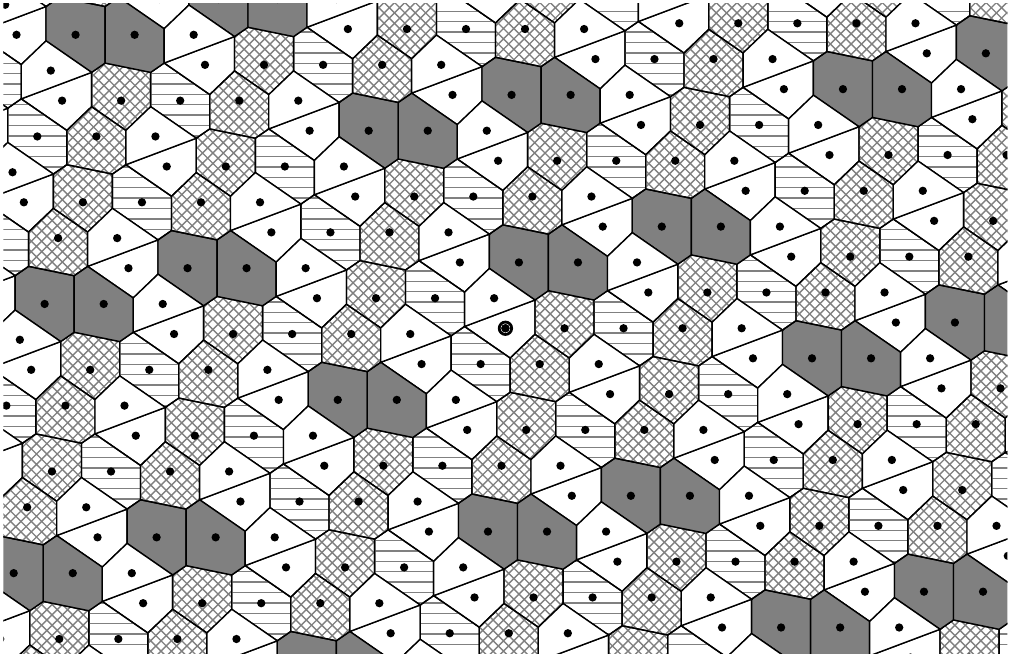}
\caption{Part of the Voronoi tessellation of
 $X^2(\gamma)=\Sigma(\Omega)$, where $\Omega=[0,\frac{2}{1-\gamma'})$
 and $\gamma=\gamma_T$ is the complex Tribonacci constant.
The point $0$ is highlighted.
}\label{fig:voronoi}
\end{figure}

\begin{figure}
\includegraphics{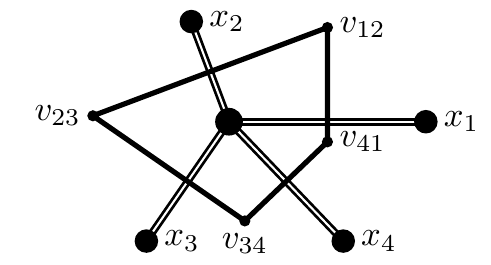}
\caption{One of the protocells of $X^2(\gamma)$.}
\label{fig:T1}
\end{figure}

To compute $L_2(\gamma)$, we first determine the value of $\Delta(\T)$ for all protocells.
By definition, $\Delta(\T)$ is twice the maximal distance from $0$ to the vertices of $\T$.
The vertices of the protocell are points $v_{ij}$ satisfying that $\abs{x_i-v_{ij}}=\abs{x_j-v_{ij}}=\abs{0-v_{ij}}$,
 see Figure~\ref{fig:T1}.
This is the same condition as \eqref{eq:vij-cond}, thus the points $v_{ij}$ are given by~\eqref{eq:vij-xij}.
Therefore we have
\begin{gather*}
	\abs{v_{12}}=\frac12\abs[\Big]{\frac{\gamma^{-2}(1-\gamma^{-2})}{\Im(\gamma^{-2})}}
	\approx 0.692
,\quad
	\abs{v_{23}}=\frac12\abs[\Big]{\frac{\gamma^{-2}(1-\gamma^{-1})}{\Im(\gamma^{-1})}}
	\approx 0.692
,\\
	\abs{v_{34}}=\abs{v_{41}}=\frac12\abs[\Big]{\frac{\gamma^{-1}(1+\gamma^{-1})}{\Im(\gamma^{-1})}}
	\approx 0.510
.\end{gather*}
Numerically, it seems that the first two values are equal.
To see that this is true, we only have to check that $\abs{1+\gamma^{-1}}=2\abs{\Re(\gamma^{-1})}$,
 because $\frac{\Im(z^2)}{\Im z}=2\Re z$ for any non-real $z\in\C$.
Since $\gamma^{-1}$ and $\conj\gamma^{-1}$ are the Galois conjugates of $\beta$ root of $Y^3-Y^2-Y-1$, 
 we have $\gamma^{-1}\conj\gamma^{-1}=1/\beta$ and $\gamma^{-1}+\conj\gamma^{-1}=1-\beta$ by Vieta's formulas.
Now we easily verify that the numbers $\abs{1+\gamma^{-1}}^2=(1+\gamma^{-1})(1+\conj\gamma^{-1})$
 and $4\abs{\Re(\gamma^{-1})}^2=(\gamma^{-1}+\conj\gamma^{-1})^2$ are equal.
We can further simplify
\[
	\abs{v_{23}}^2
	= \frac14\frac{\gamma^{-2}\conj\gamma^{-2}(1-\gamma^{-1})(1-\conj\gamma^{-1})}
		{\bigl(\frac{1}{2\ii}(\gamma^{-1}-\conj\gamma^{-1})\bigr)^2}
	= \beta\frac{\beta^2-1}{3\beta^2-1}
,\]
 because we see that the left-hand side is a symmetric rational function in $\gamma^{-1},\conj\gamma^{-1}$,
 therefore Vieta's formulas can be used to rewrite it in $\beta$'s.

Whence, for the 1st protocell, the maximal distance is $\Delta(\T)=2\abs{v_{23}}$.
It turns out that this is the value of $\Delta(\T)$ for all the protocells of $\Sigma(\Omega)$.
Therefore $L_2(\gamma)=\Delta(\T(x))$ for all $x\in X^2(\gamma)$ and the value is
\[
	L_2(\gamma)
	= 2\sqrt{\beta\frac{\beta^2-1}{3\beta^2-1}}\approx1.384
.\]
\end{example}

\begin{example}\label{ex:one-tribo-alt}
Let us give one more example.
We fix the same $\gamma=\gamma_T$ as before
 and we take $c=(\gamma')^{-2}=\beta^2$.
Then $p=2$ and $k=0$ satisfy the hypothesis of Lemma~\ref{lem:L}.
Therefore
\[
	L = \abs[\Big]{\frac{\gamma^2(\gamma-1)}{\Im\gamma}} \approx 3.4531
\]
 satisfies \eqref{eq:good-L}.
In this case, $\Xi$ is of size 40.
Figure~\ref{fig:one-tribo-alt} denotes the result of Algorithm~\ref{algor:Xi}.
We get $7$ different protocells.
The 4th one has $\delta(\T)=1$, while all the other ones have $\delta(\T)=\sqrt{\gamma'}$.
The value of $\Delta(\T)$ is equal to $2\sqrt{\beta\frac{\beta^2-1}{3\beta^2-1}}\approx1.384$
 for all of them.

\begin{figure}
\includeW{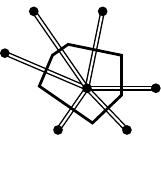}{0,\gamma'}
\includeW{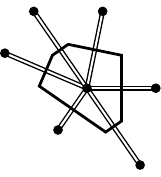}{\gamma',1}
\includeW{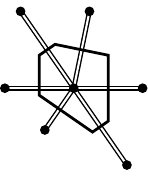}{1,1+\gamma'}
\includeW{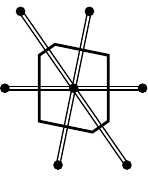}{1+\gamma',\frac{1}{\gamma'}}
\includeW{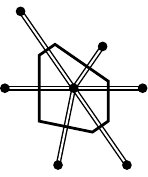}{\frac{1}{\gamma'},\frac{1}{{\gamma'}^2}-1}
\includeW{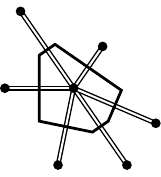}{\frac{1}{{\gamma'}^2}-1,1+\frac{1}{\gamma'}}
\includeW{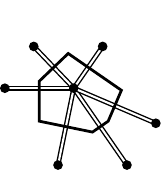}{1+\frac{1}{\gamma'},\frac{1}{{\gamma'}^2}}
\caption{Voronoi protocells (the palette) for
 $\Sigma(\Omega)$, where $\Omega=[0,\frac{1}{{\gamma'}^2})$
 and $\gamma=\gamma_T$ is the complex Tribonacci constant.
}\label{fig:one-tribo-alt}
\end{figure}

We can now run Algorithm~\ref{algor:Xi} again, using the better upper bound on $\Delta(\T)$, namely $L\approx1.384$.
This can save us a lot of steps of the algorithm: The size of $\Xi$ reduces from $40$ to $8$,
 so reduces the number of the steps.
We will use this improved value of $L$ in Section~\ref{sect:tribo}, where we study the sets
 $\Sigma([0,c))$ for all $c>0$.
\end{example}

In the two examples, we listed the palettes of $\Sigma([0,c))$ for two different values $c=\beta^2+1$ and $c=\beta^2$.
Two protocells appears in both lists.
The natural question to ask is: For which values of $c$, a given prototile occurs in the palette of $\Sigma([0,c))$?
Using Lemma~\ref{lem:L-eq}, this question can be transformed to an easier one:
 for which values of $c$, a specific $L$-patch occurs in $\Sigma([0,c))$.
Since we now treat $L$-patches for varying $c$, we denote them $\calP_L^c(x)$,
 and for convenience we denote $\bigl(\calP_L^c(x)\bigr)'\eqdef\set{z'}{z\in\calP_L^c(x)}$.
 
\begin{lemma}\label{lem:234}
Let $c_0>0$ be fixed, $c\in(0,c_0)$ and $L>0$.
Denote $-c_0\defeq w_0<w_1<\dots<w_{n-1}<w_n\eqdef c_0$ the sequence of numbers such that
\begin{equation}\label{eq:Pi}
 W\eqdef\{w_1,w_2,\dots, w_{n-1}\} = \set[\big]{z'\in\Z[\gamma']}{\abs z\leq L\text{ and }z'\in(-c_0,c_0)}
.\end{equation}
Then
\begin{enumerate}\romanenumi 

\item
 For all $x\in\Sigma([0,c))$ we have
 \[
	\calP_L^c(x)\subseteq\set{z\in\Z[\gamma]}{z'\in W}
 .\]

\item\label{enum:IJ}
 For all $x\in\Sigma([0,c))$ there exist $i,k\in\N$, $1\leq i\leq k\leq n-1$, such that
 \[
	\{w_i,w_{i+1},\dots,w_k\}=\bigl(\calP_L^c(x)\bigr)'
 .\]

\item\label{enum:UV}
 Let $1\leq i\leq k\leq n-1$.
 Then a finite set $\{w_i,w_{i+1},\dots,w_k\}$ containing $0$ equals $(\calP_L^c(x))'$ for some $x\in\Sigma([0,c))$
  if and only if
 \begin{equation}\label{eq:UV}
	w_k-w_i < c < w_{k+1}-w_{i-1}
 .\end{equation}

\item\label{enum:sym}
 For all $x\in\Sigma([0,c))$ there exists $y\in\Sigma([0,c))$ such that $\calP_L^c(y) = -\calP_L^c(x)$.

\end{enumerate}
\end{lemma}

\begin{proof}
\begin{enumerate}\romanenumi 

\item
 As $\Sigma([0,c))\subseteq \Sigma([0,c_0))$ we have $\calP_L^c(x)\subseteq \calP_L^{c_0}(x)$ and the statement follows
  from the relation~\eqref{eq:LC}.

\item
 Let $i$ and $k$ be the indices for which
 \[
	w_i=\min\bigl(\calP_L^c(x)\bigr)'
 \quad\text{and}\quad
	w_k=\max\bigl(\calP_L^c(x)\bigr)'
 .\]
 According to the relation~\eqref{eq:LC} we get
 \begin{equation}\label{eq:ST}
	0\leq x'+w_i
 \quad\text{and}\quad
	x'+w_k<c
 .\end{equation}
 Consider $w_j$ for $j\in\N$, $i<j<k$.
 Then $w_i<w_j<w_k$, whence $0\leq w_j+x'<c$.
 This implies that $w_j$ belongs to $\bigl(\calP_L^c(x)\bigr)'$ as well.

\item ($\Rightarrow$)
 Because of \eqref{eq:ST}, we have $w_k-w_i<c$.
 Since $w_{i-1}$ and $w_{k+1}$ do not belong to $\bigl(\calP_L^c(x)\bigr)'$,
  we have $x'+w_{i-1}<0$ and $x'+w_{k+1}\geq c$.
 Hence $w_{k+1}-w_{i-1}>c$.

\addtocounter{enumi}{-1} 
\item ($\Leftarrow$)
 Let $w_{i-1}, w_i, w_k, w_{k+1}$ satisfy \eqref{eq:UV}.
 As $\Z[\gamma']$ is dense in $\R$, there exists $u\in(w_{i-1},w_i)$ such that
  $u\in\Z[\gamma']$ and $u+c\in(w_k,w_{k+1})$.
 Put $x'\eqdef -u$.
 Then
 \[
	x'+w_{i-1}<0<x'+w_i<x'+w_k<c<x'+w_{k+1}
 .\]
 Since $w_i\leq0\leq w_k$, we have that $0<x'<c$, therefore $x\in\Sigma([0,c))$.
 We conclude from item~(\ref{enum:IJ}) that $\{w_i,w_{i+1},\dots,w_k\}=\bigl(\calP_L^c(x)\bigr)'$.

\item
 Since $W$ is a centrally symmetric set, i.e, $W=-W$, we have that $w_j=w_{n-j}$ for all $0\leq j\leq n$.
 Then \eqref{eq:UV} is equivalent to
 \[
 	w_{n-i} - w_{n-k} < c < w_{n-i+1} - w_{n-k-1}
 .\]
 According to item (\ref{enum:UV}), the set $\{w_i,\dots,w_k\}$ is an $L$-patch for some $x\in\Sigma([0,c))$
  if and only if $\{-w_k,\dots,-w_i\}$ is an $L$-patch for some $y\in\Sigma([0,c))$.
\qedhere

\end{enumerate}
\end{proof}

Inequality \eqref{eq:UV} answers our question.
To any $L$-patch, we can assign an open interval such that this patch occurs in $\Sigma([0,c))$
 if and only if $c$ lies in this interval.
This fact has an important consequence:
 for any given set of $L$-patches,
 the range of~$c$ such that these patches are precisely the $L$-patches of $\Sigma([0,c))$
 is an intersection of intervals and complements of intervals.
As before, the result on $L$-patches implies the following result on palettes.

\begin{corollary}\label{corol:Theta}
Let $b_0,c_0\in\R$ satisfy that $0<b_0<c_0$.
Denote by $\Pal(\Omega)$ the palette of $\Sigma(\Omega)$, i.e., the set of all protocells of $\Sigma(\Omega)$.
Then there exists a finite sequence $b_0\defeq\theta_0<\theta_1<\dots<\theta_{N-1}<\theta_N\eqdef c_0$
 such that the mapping
\[
	c\mapsto\Pal\bigl([0,c)\bigr)
\]
 is constant on each of the intervals $(\theta_{j-1},\theta_j)$ for $j=1,\dots,N$.
\end{corollary}

\begin{proof}
Consider $L$ satisfying \eqref{eq:good-L} for $\Sigma=\Sigma([0,b_0))$.
For $W$ given by \eqref{eq:Pi} find $\theta_1<\dots<\theta_{N-1}$ such that
\begin{equation}\label{eq:Theta}
	\Theta\eqdef(W-W)\cap(b_0,c_0) = \{\theta_1,\dots,\theta_{N-1}\}
.\end{equation}
Let $c,d\in(b_0,c_0)$ and suppose that the palette of $\Sigma([0,c))$ does not coincide with the palette of $\Sigma([0,d))$.
Without loss of generality there exists an $L$-patch of $x\in\Sigma([0,c))$ that is not an $L$-patch of any $y\in\Sigma([0,d))$.
This means that $c$ satisfies inequalities \eqref{eq:UV} for some indices $i,k$, whereas $d$ does not satisfy them.
This fact implies that $c$ and $d$ are separated by a point $w_k-w_i\in W-W$.
\end{proof}

The previous corollary says that there exist only finitely many palettes for $\Sigma([0,c))$ with $c\in[b_0,c_0)$.
The following algorithm determines them:

\pagebreak 

\begin{algor}\label{algor:Theta}\leavevmode
\begin{enumerate}
\item[\textbullet]
 Input: $\gamma$ satisfying \eqref{eq:cmplx-gamma}, $0<b_0<c_0$, $L$ satisfying \eqref{eq:good-L} for $\Omega=[0,b_0)$,
  e.g.\@ given by \eqref{eq:L}.
\item[\textbullet]
 Output: All possible palettes $\Pal(\Omega)$ of $\Sigma(\Omega)$ for $\Omega=[0,c)$ and $b_0\leq c<c_0$.
\item Compute the set $\Theta=\{\theta_1<\dots<\theta_{N-1}\}$ given by \eqref{eq:Theta}.
\item \label{algor:Theta:enum:for} Using Algorithm \ref{algor:Xi}, compute the palettes $\Pal(\Omega)$ for all $\Omega=[0,c)$ with
\(
	c=b_0, \tfrac{b_0+\theta_1}{2}, \theta_1,
	\dots, \tfrac{\theta_{N-2}+\theta_{N-1}}{2}, \theta_{N-1}, \tfrac{\theta_{N-1}+c_0}{2}
.\)
\item Remove possible duplicates in the list of palettes.
\end{enumerate}
\end{algor}

In Corollary~\ref{corol:Theta} and Algorithm~\ref{algor:Theta}, the assumption $b_0>0$ is very important,
 because there exist infinitely many $c\in(0,c_0)$ with different palettes.
However, these palettes cannot differ too much.
In fact, the self-similarity property (see~Proposition~\ref{prop:self-sim}) guarantees that the palette for the window $[0,\gamma'c)$
 differs from the palette for $[0,c)$ only by a scaling factor $\gamma$.
Therefore the knowledge of the palettes for $c\in[\gamma'c_0,c_0)$
 is sufficient for the description of all palettes.

\begin{remark}\label{rem:sym}
As a consequence of item (\ref{enum:sym}) of Lemma~\ref{lem:234},
 the list of $L$-patches for $\Sigma([0,c))$ is invariant under rotation by 180\degree\@.
Therefore the palette $\Pal([0,c))$ is invariant as well.
Figures~\ref{fig:one-tribo} and~\ref{fig:one-tribo-alt} witness this phenomenon.
\end{remark}

\section{Complex Tribonacci number exploited. Proof of Theorem~\ref{thm:tribo}}\label{sect:tribo}

In this section, we describe the details of the proposed workflow on an example --- the complex Tribonacci base $\gamma=\gamma_T$.
We aim at the proof of Theorem~\ref{thm:tribo}.
As usual, $\beta\eqdef \gamma\conj\gamma=1/\gamma'$.
The theorem will be proved by combining the self-similarity property in Proposition~\ref{prop:self-sim}
 and the following result:

\begin{table}
\centering
\begin{tabular}{@{\,}>$c<$@{\,}
 *{2}{c@{\,}}c@{}
 *{2}{c@{\,}}c@{}
 *{1}{c@{\,}}c@{}
 *{1}{c@{\,}}c@{}
 *{0}{c@{\,}}c@{}
 *{0}{c@{\,}}c@{}
 *{0}{c@{\,}}c@{}
 *{0}{c@{\,}}c@{\,}
 }\toprule
\text{\bfseries Interval for $c$} & \multicolumn{14}{c}{\bfseries The palette of $\Sigma(\Omega)$, where $\Omega=[0,c)$}
	\\\midrule[\heavyrulewidth]
	\beta^2 &
 & & & & & &\Tvi&\Tviii& & &\Txi& & &\Txiv
	\\\midrule
	(\beta^2,2\beta) &
 & & &\Tiv& & &\Tvi&\Tviii& & &\Txi& & &\Txiv
	\\\midrule
	(2\beta,\beta+2) &
 & & &\Tiv& & &\Tvi&\Tviii& & &\Txi& &\Txiii&
	\\\midrule
	(\beta+2,\beta^2+1) &
 & & &\Tiv& & &\Tvi& &\Tix& &\Txi& &\Txiii&
	\\\midrule
	(\beta^2+1,2\beta+1) &
 &\Tii& &\Tiv&\Tv& & & &\Tix& &\Txi& &\Txiii&
	\\\midrule
	(2\beta+1,\beta^2+\beta) &
 &\Tii& &\Tiv&\Tv& & & &\Tix&\Tx& & &\Txiii&
	\\\midrule
	(\beta^2+\beta,\beta^2+2) &
\Ti&\Tii& &\Tiv&\Tv&\Tvii& & & &\Tx& & &\Txiii&
	\\\midrule
	(\beta^2+2,2\beta+2) &
\Ti&\Tii&\Tiii& &\Tv&\Tvii& & & &\Tx& & &\Txiii&
	\\\midrule
	(2\beta+2,\beta^3) &
\Ti&\Tii&\Tiii& &\Tv&\Tvii& & & &\Tx& &\Txii& &
	\\\midrule[\heavyrulewidth]
\text{\bfseries Tile}
	&$ \!\frac1\gamma\T_4\! $&$ \T_1 $&$ \frac1\gamma\T_5 $&$ \T_2 $&$ \T_3
		$&$ \frac1\gamma\T_8 $&$ \T_4 $&$ \T_5 $&$ \T_6 $&$ \T_7 $&$ \T_8
		$&$ \frac1\gamma\T_{10} $&$ \T_9 $&$ \T_{10} $
	\\\midrule
\text{\bfseries Value of $\delta$}
	& $\frac{1}{\beta}$ & $\frac{1}{\beta}$ & $\frac{1}{\beta}$ & $\frac{1}{\beta}$ & $\frac{1}{\beta}$ & $\frac{1}{\beta}$ & $\frac{1}{\sqrt{\beta}}$ & $\frac{1}{\sqrt{\beta}}$ & $\frac{1}{\sqrt{\beta}}$ & $\frac{1}{\sqrt{\beta}}$ & $\frac{1}{\sqrt{\beta}}$ & $\frac{1}{\sqrt{\beta}}$ & $\frac{1}{\sqrt{\beta}}$ & 1
	\\\midrule
\text{\bfseries Value of $\Delta$}
	& $A$ & $B$ & $A$ & $B$ & $B$ & $A$ & $B$ & $B$ & $B$ & $B$ & $B$ & $A$ & $B$ & $B$
	\\\midrule
\text{\bfseries Value of $\Deltax $}
	& $1$ & $1$ & $1$ & $1$ & $1$ & $1$ & $\sqrt{\beta}$ & $\sqrt{\beta}$ & $\sqrt{\beta}$ & $\sqrt{\beta}$ & $\sqrt{\beta}$ & $1$ & $\sqrt{\beta}$ & $\sqrt{\beta}$
	\\\midrule[\heavyrulewidth]
\end{tabular}
\caption[The protocells for the complex Tribonacci constant for an arbitrary window $\Omega$.]{
The protocells for the complex Tribonacci constant for windows $\Omega=[0,c)$ with $c\in[\beta^2,\beta^3)$.
We put $A\eqdef2\sqrt{\frac{\beta^2-1}{3\beta^2-1}}$ and $B\eqdef A\sqrt\beta$.
Each tile in the list appears rotated by 180\degree\@ as well,
 we omit these to make the table shorter; see Remark~\ref{rem:sym}.
For a~cut-point $\theta_i$, the palette is the intersection of the palettes
 for the surrounding intervals, for instance
 $\Pal([0,\beta^2+1))=\{\T_2,\T_6,\T_8,\allowbreak\T_9,\allowbreak-\T_8,-\T_6,-\T_2\}$.
}\label{tab:pal-tribo}
\end{table}

\begin{proposition}\label{prop:dD-beta2}
Let $\Omega=[0,c)$ with $c\in(\beta^2,\beta^3)$, where $\beta\eqdef1/\gamma'$ and $\gamma$ is the complex Tribonacci constant.
Denote $\Sigma\eqdef\Sigma(\Omega)$.
Then
\begin{equation}\label{eq:dD-beta2}
	\min_{x\in\Sigma} \delta\bigl(\T(x)\bigr) = 1/\beta
\quad\text{and}\quad
	\max_{x\in\Sigma} \Delta\bigl(\T(x)\bigr) = 2\sqrt\beta\sqrt{\frac{\beta^2-1}{3\beta^2-1}}
.\end{equation}
\end{proposition}

\begin{proof}
We put $b_0\eqdef\beta^2$ and $c_0\eqdef\beta^3$.
In Example \ref{ex:one-tribo-alt} we have shown that
 $L=2\sqrt{\beta}\sqrt{\frac{\beta^2-1}{3\beta^2-1}}\approx1.384$
 satisfies \eqref{eq:good-L} for $\Omega=[0,b_0)$.
Using this $L$, we run Algorithm \ref{algor:Theta}.
The first step of the algorithm computes the set $\Theta$ defined by~\eqref{eq:Theta}.
This $\Theta$ has $14$ elements, they are drawn in the following picture:
\[
	\includegraphics{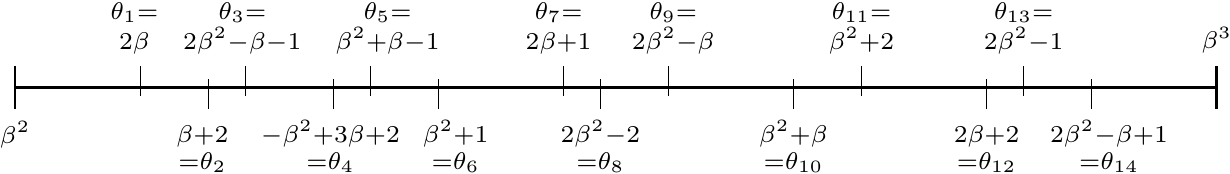}
\]
The number of cases in step \ref{algor:Theta:enum:for} of the algorithm is then $30$.
This means that we have to run Algorithm \ref{algor:Xi} exactly $30$ times to obtain all possible palettes.
Amongst the $30$ cases mentioned above, there are some duplicates, and we end up with only $16$ cases:
 $8$ cases correspond to cut-points $\theta_0,\theta_1,\theta_2,\theta_6,\theta_7,\theta_{10},\theta_{11},\theta_{12}$,
 the other $8$ cases correspond to the open intervals between the cut-points.
Moreover, we observe that for each cut-point $\theta_i$, the palette $\Pal([0,\theta_i))$
 is the intersection of the palettes of the two surrounding intervals.
All the palettes for the intervals are depicted in~Table~\ref{tab:pal-tribo}.

At the bottom of the table, the values of $\delta(\T)$ and $\Delta(\T)$ are written out
 for each protocell.
It turns out that every row of the table but the special case $c=\beta^2$
 has the minimal value of $\delta$ equal to $1/\beta\approx0.5437$ and the maximal value of $\Delta$ equal to
 $2\sqrt\beta\sqrt{\frac{\beta^2-1}{3\beta^2-1}}\approx 1.3843$.
\end{proof}

We recall that two of the runs of Algorithm~\ref{algor:Xi}, for $c=\frac{2}{1-\gamma'}=\beta^2+1\in\Theta$, i.e., for $X^2(\gamma)$, and for $c=\beta^2$ are explained in Examples~\ref{ex:one-tribo}
 and \ref{ex:one-tribo-alt} (cf.~also Figures~\ref{fig:one-tribo} and \ref{fig:one-tribo-alt}).
We have drawn a part of the Voronoi tessellation of $X^2(\gamma)$ in Figure~\ref{fig:voronoi}.

\begin{proof}[Proof of Theorem~\ref{thm:tribo}]
The theorem is a direct corollary of Proposition~\ref{prop:self-sim}, Theorem~\ref{thm:CnP}, Proposition~\ref{prop:dD-beta2}
 and of the following two facts:
\begin{itemize}

\item It cannot happen that $c=m/(1-\gamma')=(\gamma\conj\gamma)^k=\beta^k$ for some $m\geq1$ and $k\in\Z$.
 For, assume on the contrary that the last equation holds.
 Then $\beta^k\geq m$ and so $k\geq1$.
 Moreover, $k\geq3$, since $\gamma$ is cubic, and we have, by Galois isomorphism, that $m\gamma^k=1-\gamma$.
 The relation $\abs{m\gamma^k}\geq\abs{\gamma^3}>\abs{1-\gamma}$ yields a contradiction.

\item
If $\T$ is a Voronoi protocell in $\Sigma(\Omega)$ then $\gamma^k\T$ is a Voronoi protocell
 in $\gamma^k\Sigma(\Omega)=\Sigma((\gamma')^k\Omega)$ for any $k\in\Z$.
For any $m\in\N$ there exists $k\in\Z$ such that $(\gamma')^k\frac{m}{1-\gamma'}\in(\beta^2,\beta^3)$.
\qedhere

\end{itemize}
\end{proof}

\begin{remark}
Let us point out that for a real base $\beta$ the characteristic $L_m(\beta)$ given by~\eqref{eq:Lmb}
 is not influenced by gaps $x_{k+1}-x_k$ occurring only in a bounded piece of the real line.
Therefore in general the value $L_m(\gamma)$ as we have defined for the complex number $\gamma$
 is not the precise analogy to $L_m(\beta)$.
Nevertheless, if the set $X^m(\gamma)$ is repetitive (i.e., any patch occurs infinitely many times),
 which is our case, then omitting configurations in a bounded area of the plane plays no~role.{\looseness1\par}
\end{remark}

\begin{figure}
\centering
\includegraphics[width=0.3\linewidth]{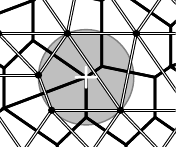}
\caption{Part of Voronoi (in solid lines) and Delone (in double lines) tessellations of $X^2(\gamma)$
 for $\gamma=\gamma_T$ the complex Tribonacci constant.
The white cross is a vertex of the Voronoi tessellation, and at the same time, it is a center of the gray circle,
 on which four points of $X^2(\gamma)$ lie.}
\label{fig:ex-delone}
\end{figure}

\section{Delone tessellation --- dual to Voronoi tessellation}\label{sect:delone}

From Voronoi tessellation we can construct its dual tessellation:
Let $\Sigma\subseteq\C$ be a Delone set.
Consider a planar graph in $\C$ whose vertices are elements of the set $\Sigma$
 and edges are line segments connecting $x,y\in\Sigma$ where
 $x$ and $y$ are \defined{neighbors}, i.e., their Voronoi cells $\T(x)$ and $\T(y)$ share a side.
This graph divides the complex plane into faces; these faces are called \defined{Delone tiles}.
The collection of Delone tiles is the \defined{Delone tessellation} of $\Sigma$.

All vertices of a Delone tile lie on a circle;
 its center is a vertex of the Voronoi tessellation.
This is illustrated in Figure~\ref{fig:ex-delone}, which shows a small part of the set $X^2(\gamma)$,
 where $\gamma$ is the complex Tribonacci constant;
 the quadrilateral is inscribed in the circle.
The white cross marks the center of the circle
 and it is a common vertex of four Voronoi cells.

The minimal distance $\inf_{x\in\Sigma} \delta(\T(x))$ is equal
 to the shortest edge in the Delone tessellation.
On the other hand, the longest edge in the Delone tessellation
 is (in~general) shorter than $\sup_{x\in\Sigma} \Delta(\T(x))$.
Therefore, for a point $x\in\Sigma(\Omega)$ we can define
\[
	\Deltax \bigl(\T(x)\bigr) \eqdef  \max \set[\big]{\abs{x-y}}{\text{$y$ is a neighbor of $x$ in $\Sigma$}}
\]
 and study its maximum over all points $x\in\Sigma$.

We can apply this to the sets $X^m(\gamma)$.
We define
\[
	L^*_m(\gamma) = \Lx_m(\gamma)\eqdef \!\sup_{x\in X^m(\gamma)}\! \Deltax \bigl(\T(x)\bigr)
\]
 if $X^m(\gamma)$ is Delone, and $\Lx_m(\gamma)=+\infty$ otherwise.
When $X^m(\gamma)$ is a cut-and-project set, we know that it has a finite local complexity
 and therefore finitely many different Delone tiles up to translation.

In the case of the complex Tribonacci base, the shapes of all Delone tiles of $X^2(\gamma)$ are depicted in Figure~\ref{fig:delone-tiles}.
From Table~\ref{tab:pal-tribo} we get the following result:

\begin{figure}
\includeD{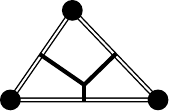}\!\!\!\includeD{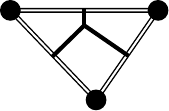}\quad\includeD{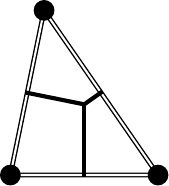}\!\!\!\!\!\includeD{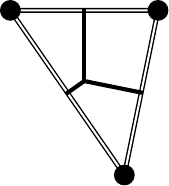}\quad\includeD{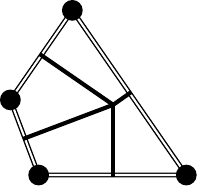}\!\!\!\!\!\includeD{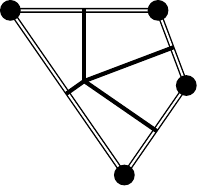}
\caption{Delone tiles of the set $X^2(\gamma)$, where $\gamma=\gamma_T$ is the complex Tribonacci constant.}
\label{fig:delone-tiles}
\end{figure}

\begin{theorem}\label{thm:Deltax}
With the hypothesis of Theorem~\ref{thm:tribo}, we have:
\[
	L^*_m(\gamma)=\abs{\gamma}^{3-k}
.\]
\end{theorem}

\section{Comments and open problems}\label{sect:concl}

This paper treated a family of cubic complex Pisot units $\gamma$ --- such ones that the real number $1/\gamma'$
 is positive and satisfies Property~(F).
We used the concept of cut-and-project sets to study the properties of the sets $X^m(\gamma)$.
However, there are other cases where it might be possible to use this concept:

\begin{enumerate}

\item
 We can consider a different perspective of the Tribonacci constant.
 Let $\gamma$ be the complex root of $Y^3+Y^2+Y-1$, and put $\beta\eqdef1/\gamma'$.
 Both $\gamma$ and $-\gamma$ are complex Pisot units.

 It was shown by V\'avra \cite{vavra_2013} that the real Tribonacci constant $\beta$ has the so-called Property~($-$F).
 Shortly speaking, all numbers from $I\cap\Z[-1/\beta]=I\cap\Z[\beta]$,
  where $I\eqdef(\frac{-\beta}{\beta+1},\frac{1}{\beta+1})$, have a finite expansion of the form
  $\frac{a_1}{-\beta}+\frac{a_2}{\beta^2}+\frac{a_3}{-\beta^3}+\dotsb$ with $a_j\in\{0,1\}$.
 From this, we can show that $X^m(-\gamma)$ is a cut-and-project set for arbitrary $m\geq1$.
 The idea goes along the lines of the proof of Theorem~\ref{thm:CnP}.

\item
 Consider any real Pisot unit $\beta$ of degree $n$.
 Let $\gamma=\ii\sqrt\beta$.
 Then $\gamma$ is a complex Pisot unit of degree $2n$,
  its Galois conjugates are $\conj\gamma$ and $\pm\ii\sqrt{\beta'}$ for $\beta'$ conjugates of $\beta$.

 Clearly $X^m(\gamma)=X^m(-\beta)+\ii\sqrt\beta X^m(-\beta)$.
 Therefore the Voronoi cells of $X^m(\gamma)$ are rectangles.
 Values $\ell_m(\gamma)$ and $L_m(\gamma)$ can be easily obtained from the minimal and maximal distances in $X^m(-\beta)$.
 In the case $n=2$, relations between $X^m(-\beta)$ and cut-and-project sets in dimensions $d=e=1$ were established in~\cite{mpp_preprint},
 implying that $X^m(\gamma)$ is related to cut-and-project sets in dimensions $d=e=2$.

 Let us note that Za\"imi \cite{zaimi_2004} evaluated $\ell_m(\gamma)$ for $\gamma=\ii\sqrt{\beta}$, $m=\qfl{\beta^2}$
  and $\beta>1$ the root of $Y^2-aY-a$, $a\in\N$.

\item\label{enum:non-F}
 In the cubic case, we can weaken the hypothesis of Theorem~\ref{thm:CnP}.
 For a fixed $m$, the Property~(F) can be replaced by the assumption that all numbers from $\Z[\beta]\cap[0,1)$
  have a finite $\beta$-representation over the alphabet $\{0,1,\dots,m\}$, where we denote $\beta\eqdef 1/\gamma'>1$.
 Under such an assumption, $X^m(\gamma)$ is a cut-and-project set.

 Akiyama, Rao and Steiner~\cite{akiyama_rao_steiner_2004}
  described precisely the set of purely periodic expansions of points from $\Z[\beta]$.
 They have shown that all of them are of the form $.ccc\dotsm=.c^\omega$, where $0\leq c<\qfl\beta$ and $(a+b)\mid c$.
 Since all numbers from $\Z[\beta]\cap[0,1)$ have finite or periodic $\beta$-expansions~\cite{schmidt_1980}
  (and the only periods are therefore the ones mentioned above),
  it is satisfactory to find $m_1$ such that the number $.(a+b)^\omega$ has a finite representation over the alphabet $\{0,\dots,m_1\}$.
 Under this hypothesis, all numbers from $\Z[\beta]\cap[0,1)$ have a finite representation over the alphabet $\{0,\dots,m\}$
  for all $m\geq m_1\qfl{\frac{\beta}{a+b}}$.
 We were not able to establish the hypothesis in all cases. We list some cases in Table~\ref{tab:non-F}.

\item
 Quartic Pisot units $\gamma$  with $\abs\gamma \in (1,2)$ are treated by Dombek, Mas\'akov\'a and Ziegler
  in~\cite{dombek_masakova_ziegler_2013}.
 The authors study the question of whether every element of the ring $\mathbb{Z}[\gamma]$ of integers of $\mathbb{Q}(\gamma)$
  can be written as a sum of distinct units.
 If the only units on the unit circle are $\pm1$, then the question can be interpreted as Property (F) over the alphabet $\{-1,0,1\}$.
 Therefore the concept of cut-and-project sets can be applied to these quartic bases and symmetric alphabets as well.

\end{enumerate}

\begin{table}
\begin{tabular}{>$c<$ >$c<$ >$c<$ >$l<$}\toprule
b  & a      & m_1   & \multicolumn{1}{c}{\text{\bfseries Representation of $.(a+b)^\omega$}}
\\\midrule[\heavyrulewidth]
-2 & \geq 3 & 2a-2  & .(a-3)(2a-2)(a-3)(0)(1) \\\midrule[\heavyrulewidth]
-3 & \geq 7 & 3a-6  & .(a-4)(2a-5)(3a-6)(a-7)(0)(1) \\\cmidrule{2-4}
   &    = 6 & 10    & .(2)(7)(10)(10)(0)(0)(1) \\\cmidrule{2-4}
   &    = 5 & 9     & .(0)(9)(9)(5)(0)(0)(1) \\\cmidrule{2-4}
   &    = 4 & 7     & .(0)(2)(6)(7)(0)^3(1) \\\midrule[\heavyrulewidth]
-4 & \geq 8 & 8a-11 & .(a-5)(2a-11)(8a-11)(4a-31)(a-8)(0)(1) \\\cmidrule{2-4}
   &    = 7 & 39    & .(0)(16)(39)(27)(0)^3(1) \\\cmidrule{2-4}
   &    = 6 & 47    & .(0)(3)(44)(47)(0)^4(1) \\\midrule[\heavyrulewidth]
\end{tabular}
\caption[List of pairs of $a,b$ such that $X^m(\gamma)$ is a cut-and-project set.]{List of pairs of $a,b$ such that $X^m(\gamma)$ is a cut-and-project set,
 where $\gamma$ is the non-real root of $Y^3+bY^2+aY-1$ and $m\geq m_1\qfl{\frac{1/\gamma'}{a+b}}$.}
\label{tab:non-F}
\vspace*{-3ex}
\end{table}

Let us conclude with several open questions:
\begin{enumerate}\Alphenumi 

\item\label{enum:Q:mF}
 Is it true that all real cubic Pisot units $\beta$ with a complex conjugate satisfy the following:
 There exists $m\in\N$ such that all numbers from $\Z[\beta]\cap[0,1)$ have finite
  $\beta$-representation over the alphabet $\{0,\dots,m\}$?

\item
 Which real cubic unit bases $-\beta$, other than minus the Tribonacci constant, satisfy Property ($-$F)?
 Which $-\beta$ satisfy the statement proposed in Question~(\ref{enum:Q:mF})?

\item
 It is well known that, in the real case, $X^m(\beta)$ is a relatively dense set in $\R_+$
  if and only if $m>\beta-1$.
 Can we state analogous result in the complex case?
 In particular, is $X^m(\gamma)$ relatively dense set in $\C$ for all $m>\abs{\gamma}^2-1$?

 Can the complex modification of the Feng's result~\cite{feng_2013} be proved,
  namely that $\ell_m(\gamma)=0$ if and only if $m>\abs\gamma^2-1$ and $\gamma$ is not a complex Pisot number?

\end{enumerate}

\section*{Acknowledgements}
We are grateful to the unknown referee, who helped correct some of our calculations,
 and whose careful reading of our paper significantly improved the presentation of the results.

We would like to thank Wolfgang Steiner for our fruitful discussions.

This work was supported by
 Grant Agency of the Czech Technical University in Prague grant SGS14/205/OHK4/3T/14,
 Czech Science Foundation grant 13-03538S,
 and ANR/FWF project ``FAN -- Fractals and Numeration'' (ANR-12-IS01-0002, FWF grant I1136).


\begin{thebibliography}{MPZ03b}

\bibitem[Aki00]{akiyama_2000}
Shigeki Akiyama, \emph{Cubic {P}isot units with finite beta expansions},
  Algebraic number theory and {D}iophantine analysis ({G}raz, 1998), de
  Gruyter, Berlin, 2000, pp.~11--26. \MR{1770451 (2001i:11095)}

\bibitem[ARS04]{akiyama_rao_steiner_2004}
Shigeki Akiyama, Hui Rao, and Wolfgang Steiner, \emph{A certain finiteness
  property of {P}isot number systems}, J. Number Theory \textbf{107} (2004),
  no.~1, 135--160. \MR{2059954 (2005g:11135)}

\bibitem[BH02]{borwein_hare_2002}
Peter Borwein and Kevin~G. Hare, \emph{Some computations on the spectra of
  {P}isot and {S}alem numbers}, Math. Comp. \textbf{71} (2002), no.~238,
  767--780. \MR{1885627 (2003a:11135)}

\bibitem[BH03]{borwein_hare_2003}
\bysame, \emph{General forms for minimal spectral values for a class of
  quadratic {P}isot numbers}, Bull. London Math. Soc. \textbf{35} (2003),
  no.~1, 47--54. \MR{1934431 (2003i:11154)}

\bibitem[Bug96]{bugeaud_1996}
Yann Bugeaud, \emph{On a property of {P}isot numbers and related questions},
  Acta Math. Hungar. \textbf{73} (1996), no.~1-2, 33--39. \MR{1415918
  (98c:11113)}

\bibitem[DMZ13]{dombek_masakova_ziegler_2013}
Daniel Dombek, Zuzana Mas{\'a}kov{\'a}, and Volker Ziegler, \emph{On distinct
  unit generated fields that are totally complex}, 2013, submitted, 14~pp.,
  \href{http://arxiv.org/abs/1403.0775}{\ttfamily arXiv:1403.0775}.

\bibitem[EJK90]{erdos_joo_komornik_1990}
Paul Erd{\H{o}}s, Istv{\'a}n Jo{\'o}, and Vilmos Komornik,
  \emph{Characterization of the unique expansions
  {$1=\sum^\infty_{i=1}q^{-n_i}$} and related problems}, Bull. Soc. Math.
  France \textbf{118} (1990), no.~3, 377--390. \MR{1078082 (91j:11006)}

\bibitem[EJK98]{erdos_joo_komornik_1998}
\bysame, \emph{On the sequence of numbers of the form
  {$\epsilon_0+\epsilon_1q+\cdots+\epsilon_nq^n,\ \epsilon_i\in\{0,1\}$}}, Acta
  Arith. \textbf{83} (1998), no.~3, 201--210. \MR{1611185 (99a:11022)}

\bibitem[Fen13]{feng_2013}
De-Jun Feng, \emph{On the topology of polynomials with bounded integer
  coefficients}, 2013, submitted, 13~pp.,
  \href{http://arxiv.org/abs/1109.1407}{\ttfamily arXiv:1109.1407}.

\bibitem[FW02]{feng_wen_2002}
De-Jun Feng and Zhi-Ying Wen, \emph{A property of {P}isot numbers}, J. Number
  Theory \textbf{97} (2002), no.~2, 305--316. \MR{1942963 (2003i:11155)}

\bibitem[KLP00]{komornik_loreti_pedicini_2000}
Vilmos Komornik, Paola Loreti, and Marco Pedicini, \emph{An approximation
  property of {P}isot numbers}, J. Number Theory \textbf{80} (2000), no.~2,
  218--237. \MR{1740512 (2000k:11116)}

\bibitem[Kom02]{komatsu_2002}
Takao Komatsu, \emph{An approximation property of quadratic irrationals}, Bull.
  Soc. Math. France \textbf{130} (2002), no.~1, 35--48. \MR{1906191
  (2003b:11063)}

\bibitem[Moo97]{moody_1997}
Robert~V. Moody, \emph{Meyer sets and their duals}, The mathematics of
  long-range aperiodic order ({W}aterloo, {ON}, 1995), NATO Adv. Sci. Inst.
  Ser. C Math. Phys. Sci., vol. 489, Kluwer Acad. Publ., Dordrecht, 1997,
  pp.~403--441. \MR{1460032 (98e:52029)}

\bibitem[MPP14]{mpp_preprint}
Zuzana Mas{\'a}kov{\'a}, Kate{\v r}ina Pastir{\v c}{\'a}kov{\'a}, and Edita
  Pelantov{\'a}, \emph{Description of spectra of quadratic {P}isot units},
  2014, submitted, 22~pp., \href{http://arxiv.org/abs/1402.1582}{\ttfamily
  arXiv:1402.1582}.

\bibitem[MPZ03a]{masakova_patera_zich_2003_i}
Zuzana Mas{\'a}kov{\'a}, Ji{\v r}{\'\i} Patera, and Jan Zich,
  \emph{Classification of {V}oronoi and {D}elone tiles in quasicrystals. {I}.
  {G}eneral method}, J. Phys. A \textbf{36} (2003), no.~7, 1869--1894.
  \MR{1960699 (2004a:52041)}

\bibitem[MPZ03b]{masakova_patera_zich_2003_ii}
\bysame, \emph{Classification of {V}oronoi and {D}elone tiles of quasicrystals.
  {II}. {C}ircular acceptance window of arbitrary size}, J. Phys. A \textbf{36}
  (2003), no.~7, 1895--1912. \MR{1960700 (2004a:52042)}

\bibitem[MPZ05]{masakova_patera_zich_2005}
\bysame, \emph{Classification of {V}oronoi and {D}elone tiles of quasicrystals.
  {III}. {D}ecagonal acceptance window of any size}, J. Phys. A \textbf{38}
  (2005), no.~9, 1947--1960. \MR{2124374 (2005j:52025)}

\bibitem[R{\'e}n57]{renyi_1957}
A.~R{\'e}nyi, \emph{Representations for real numbers and their ergodic
  properties}, Acta Math. Acad. Sci. Hungar \textbf{8} (1957), 477--493.
  \MR{0097374 (20 \#3843)}

\bibitem[Sage]{sage}
The Sage Group, \emph{{Sage}: {O}pen source mathematical software (version
  6.1.1)}, 2014, \url{http://www.sagemath.org}.

\bibitem[Sch80]{schmidt_1980}
Klaus Schmidt, \emph{On periodic expansions of {P}isot numbers and {S}alem
  numbers}, Bull. London Math. Soc. \textbf{12} (1980), no.~4, 269--278.
  \MR{576976 (82c:12003)}

\bibitem[TikZ]{tikz}
Till Tantau et al., \emph{Ti\textit{k}{Z} \& {PGF} (version 2.10)}, 2010,
  \href{http://sourceforge.net/projects/pgf}{\ttfamily http://sourceforge.net/}\allowbreak
  \href{http://sourceforge.net/projects/pgf}{\ttfamily projects/pgf}.

\bibitem[V{\'a}v14]{vavra_2013}
Tom{\'a}{\v s} V{\'a}vra, \emph{On the finiteness property of negative cubic
  {P}isot bases}, 2014, submitted, 13~pp.,
  \href{http://arxiv.org/abs/1404.1274}{\ttfamily arXiv:1404.1274}.

\bibitem[Za{\"{\i}}04]{zaimi_2004}
Toufik Za{\"{\i}}mi, \emph{On an approximation property of {P}isot numbers.
  {II}}, J. Th\'eor. Nombres Bordeaux \textbf{16} (2004), no.~1, 239--249.
  \MR{2145586 (2006f:11133)}

\end{thebibliography}

\end{document}